\documentclass[12pt]{article}
\usepackage{amsmath,amsthm,amsthm,mathrsfs}
\usepackage{fancyhdr}
\usepackage[british]{babel}
\usepackage{amssymb}
\usepackage{hyperref,cleveref}
\usepackage{latexsym}
\usepackage{bbm}
\usepackage{graphicx}
\usepackage{epstopdf}
\usepackage{float}
\usepackage{color}
\usepackage[natural]{xcolor}
\usepackage{setspace}%行距\usepackage{listings}
\usepackage{listings}
\lstdefinestyle{Python}{
    language=Python,
}

\usepackage{textgreek}%textlambda
\usepackage{float,color,fancybox,shapepar,setspace,hyperref}
\usepackage{enumitem}
\theoremstyle{plain}
\usepackage{todonotes}
\newtheorem{theorem}{Theorem}
\newtheorem{claim}{Claim}
\newtheorem{lemma}[theorem]{Lemma}

\numberwithin{theorem}{section}
\newtheorem{corollary}[theorem]{Corollary}
\newtheorem{problem}{Problem}
\newtheorem{proposition}[theorem]{Proposition}

\newtheorem{definition}[theorem]{Definition}

\newtheorem{remark}[theorem]{Remark}
\begin{document}

\title{\LARGE Local rainbow colorings of hypergraphs 
\thanks{Emails: 
{\tt zy.li@mail.sdu.edu.cn},
{\tt wcliu@sdu.edu.cn},
{\tt gwsun@mail.sdu.edu.cn},\\   \indent\indent\indent\indent\indent\indent
{\tt xiawang@mail.sdu.edu.cn},
{\tt snwei@mail.sdu.edu.cn}
}}

\author{Zhenyu Li
\quad
Weichan Liu
\quad 
Guowei Sun\\
Xia Wang
\quad
Shunan Wei
\\
\small{School of Mathematics, Shandong University, Jinan, China.}
}

\date{}
 
\maketitle
\begin{abstract}
In this paper, we generalize the concepts related to rainbow coloring to hypergraphs. Specifically, an $(n,r,H)$-local coloring is defined as a collection of $n$ edge-colorings, $f_v: E(K^{(r)}_n) \rightarrow [k]$ for each vertex $v$ in the complete $r$-uniform hypergraph $K^{(r)}_n$, with the property that for any copy $T$ of $H$ in $K^{(r)}_n$, there exists at least one vertex $u$ in $T$ such that $f_u$ provides a rainbow edge-coloring of $T$ (i.e., no two edges in $T$ share the same color under $f_u$). The minimum number of colors required for this coloring is denoted as the local rainbow coloring number $C_r(n, H)$.
 
We first establish an upper bound of the local rainbow coloring number for $r$-uniform hypergraphs $H$ consisting of $h$ vertices, that is, $C_r(n, H)= O\left( n^{\frac{h-r}{h}} \cdot h^{2r + \frac{r}{h}} \right)$. Furthermore, we identify a set of $r$-uniform hypergraphs whose local rainbow coloring numbers are bounded by a constant. A notable special case indicates that $C_3(n,H) \leq C(H)$ for some constant $C(H)$ depending only on $H$ if and only if $H$ contains at most 3 edges and does not belong to a specific set of three well-structured hypergraphs, possibly augmented with isolated vertices. We further establish two 3-uniform hypergraphs $H$ of particular interest for which $C_3(n,H) = n^{o(1)}$.
 
Regarding lower bounds, we demonstrate that for every $r$-uniform hypergraph $H$ with sufficiently many edges, there exists a constant $b = b(H) > 0$ such that $C_r(n,H) = \Omega(n^b)$. Additionally, we obtain lower bounds for several hypergraphs of specific interest.
     
\end{abstract}
\section{Introduction}

In 1993, Karchmer and Wigderson \cite{KW93} introduced the \textit{Span Program}, a computational model rooted in the principles of linear algebra. 
The Span Program has attracted widespread interest due to its higher computational efficiency compared to models such as Switching Networks and DeMorgan Formulas \cite{jukna2011extremal}. 
To establish lower bounds for the minimum possible size of the Span Program, Karchmer and Wigderson \cite{KW93} devised the Fusion Method.
Driven by this motivation, Wigderson \cite{zbMATH00524142} suggested three problems. One of these problems is shown below, which is also highlighted as a research problem  \cite[Problem 9 in Chapter 16]{jukna2011extremal}.

\begin{problem}[\cite{zbMATH00524142,jukna2011extremal}]\label{originalproblem}
Define $k$ as the smallest number such that there exist $n$ colorings, denoted as $c_1, c_2, \ldots, c_n$, of the $n$-dimensional cube $\mathbb{Q}_n:=\{0, 1\}^n$ using $k$ colors satisfying the condition that
for any triple of distinct vectors $x, y, z$, there exists at least one coordinate $i$ where not all three vectors have the same value, and the colors assigned to these vectors at the $i$th coordinate, i.e., $c_i(x), c_i(y), c_i(z)$, are all different from each other. The task is to determine or estimate the smallest possible value of $k$ for which such a set of colorings can be found.
\end{problem}

This problem is difficult, as demonstrated by Karchmer and Wigderson \cite{KW93}, who showed that $k$ inevitably increases with $n$, and more precisely $k$ is at least $\Omega\left(   \frac{\log \log^* n}{ \log \log \log^* n}  \right)$, where $\log^* n$ is the minimum number $m$ such that starting with $n$ one gets a number that does not exceed 1 by iteratively applying the function $\log_2(x)$ $m$ times. Alon and Ben-Eliezer \cite{zbMATH06056109} improved this bound to $\Omega\left(   (\frac{\log n}{ \log \log n})^{1/4}  \right)$ by considering the vectors whose Hamming weight is 2. Here, the \textit{Hamming weight} refers to the number of non-zero symbols in a string of symbols. 
We set $\mathbb{Q}_n[r]:=\{x\in \mathbb{Q}_n~|~{\rm Hamming ~weight ~of } ~x ~{\rm is~} r\}$.

In \cite{zbMATH06056109}, Alon and Ben-Eliezer constructed an auxiliary graph $G$ as follows. The vertex set of $G$ is the $n$ coordinates of $\mathbb{Q}_n[2]$, and for any $i,j\in V(G)$, $ij$ is an edge of $G$ if and only if the $n$-dimensional vector $x_{ij}$ belongs to $\mathbb{Q}_n[2]$, where $x_{ij}$ has only the positions $i$ and $j$ as $1$s and the other positions as $0$s.
Obviously, $G$ is a complete labeled graph $K_n$. The collection of colorings $c_1, c_2, \ldots, c_n$ in Problem \ref{originalproblem} corresponds to a set of $n$ edge-colorings
\[
\{f_v: E(K_n) \rightarrow [k]\ | \ v \in V(K_n)\},
\]
and any three vectors in $\mathbb{Q}_n[2]$ correspond to three distinct edges in $K_n$. Through this construction, the situation of Problem \ref{originalproblem} in $\mathbb{Q}_n[2]$ is transformed into a problem of local rainbow colorings of graphs.
Formally, Alon and Ben-Eliezer gave the following definition.
\begin{definition}[\cite{zbMATH06056109}]\label{oridef}
Define $C_2(n, H)$ as the minimum $k$ such that there exists a set of $n$ edge-colorings, $f_v: E(K_n) \rightarrow [k]$ for every $v \in V(K_n)$, where for any copy $T$ of $H$ in $K_n$, there is a vertex $u \in V(T)$ so that $f_u$ is a rainbow edge-coloring of $T$, meaning no two edges of $T$ share the same color under $f_u$. Such a set of edge-colorings is called an \emph{$(n, H)$-local coloring}.
\end{definition}

The value of $C_2(n,H)$ is a lower bound of Problem \ref{originalproblem} when $H$ consists of three edges and no isolated vertex. Using this idea,  Alon and Ben-Eliezer \cite{zbMATH06056109} established a lower bound of Problem \ref{originalproblem} by showing $C_2(n, P_3) = \Omega\left(   (\frac{\log n}{ \log \log n})^{1/4}  \right)$ as mentioned before. Additionally, they provided a specific construction showing that $C_2(n,P_3)$ does not exceed $2\lceil \sqrt{n} \rceil$. Later, Janzer and Janzer \cite{JANZER2024134} demonstrated that $C_2(n, P_3) = n^{o(1)}$. This implies that $C_2(n, H)$ may not be a polynomial in $n$, answering a question of Alon and Ben-Eliezer in the negative. In the same paper (specifically Section 4.2), Janzer and Janzer mentioned that one might give a better lower bound of Problem $\ref{originalproblem}$ than $ n^{o(1)}$ by considering the $s$-graph induced by three hyperedges $x=x_0\cup S$, $y=y_0\cup S$ and $z=z_0\cup S$. Here, $x_0$, $y_0$, $z_0$ are three edges that form $P_3$ and $S$ is a vertex set of size $s-2$, where $s$ is an integer that at least logarithmic in $n$. 
This inspires us to transform the situation of Problem \ref{originalproblem} in $\mathbb{Q}_n[r]$ into a problem of local rainbow coloring of hypergraphs.
To this end, let each coordinate within $\mathbb{Q}_n[r]$ correspond to a vertex in an $r$-uniform hypergraph $G$, and let each vector, characterized by Hamming weights of $r$, correspond to a hyperedge in $G$, where the hyperedge contains vertex $i$ if and only if the $i$th coordinate of the corresponding vector is 1. It is apparent that $G$ is a labeled complete hypergraph $K^{(r)}_n$. The collection of colorings $c_1, c_2, \ldots, c_n$ in Problem \ref{originalproblem} corresponds to a set of $n$ edge-colorings $\{f_v: E(K^{(r)}_n) \rightarrow [k]\ | \ v \in V(K_n) \}$.
This more general framework has potential applications for Problem \ref{originalproblem} and its related issues while also possessing independent interests. 
Formally, we give the following concept.

\begin{definition}
An \emph{$(n,r,H)$-local coloring} is a set of $n$ $k$-edge-colorings, $f_v: E(K^{(r)}_n) \rightarrow [k]$ for $v \in V(K^{(r)}_n)$, such that for any copy $T$ of $H$ in $K^{(r)}_n$, there exists a vertex $u \in T$ so that $f_u$ is a rainbow edge-coloring of $T$, which means that no two edges of $T$ share the same color under $f_u$. Denote by the \emph{local rainbow coloring number $C_r(n, H)$} the minimum number $k$ of colors such that there exists an $(n, r, H)$-local coloring of $K^{(r)}_n$.
\end{definition}

For 2-graphs, Alon and Ben-Eliezer \cite{zbMATH06056109} demonstrated that if $H$ contains at most 3 edges and is neither $P_3$ nor $P_3$ augmented with isolated vertices, then $C_2(n,H)$ is at most 5 for any $n$. For general upper bounds, they proved that
$C_2(n,H)=O(n^{\frac{h-2}{h}}\cdot h^4)$, where $h$ is the number of vertices in $H$.  For lower bounds, they proved that $C_2(n,H)=\Omega(n^b)$ for some positive constant $b = b(H)$ when $H$ has at least 13 edges. This bound is later lowered to 6 by Cheng and Xu \cite{zbMATH07882976} and further refined to 4 by Janzer and Janzer \cite{JANZER2024134}, with the exception of the triangle augmented with a pendant edge.

In this paper, we study similar problems in hypergraphs. Let $r$ and $h$ be some positive integers.  We are keen on identifying $r$-graphs $H$ for which $C_r(n, H)$ is bounded by a constant $C(H)$. For $r=3$, let $\mathcal{H}$ be a family of all $3$-graphs with $3$ edges except for the following three hypergraphs, possibly with some isolated vertices added,
    \begin{description}
    \item[${\rm TP}_3$] --- a 3-graph on vertex set $\{a,b,c,d,e\}$ and edge set $\big\{ \{a,b,c\}, \{b,c,d\}, \{c,d,e\}  \big\}$;
    \item[${\rm SP}_3$] --- a 3-graph on vertex set $\{a,b,c,d,e,f\}$ and edge set $\big\{ \{a,b,c\}, \{b,c,d\}, \{d,e,f\}  \big\}$; 
    \item[${\rm LC}_3$] --- a 3-graph on vertex set $\{a,b,c,d,e,f\}$ and edge set $\big\{ \{a,b,d\}, \{b,c,e\}, \{a,c,f\} \big\}$.
\end{description}(see Figures \ref{fig:tightp3}, \ref{fig:tightloosepath}, \ref{fig:loosec3}).
\begin{figure}[H]
\centering
\begin{minipage}[b]{0.3\linewidth}
\centering
\includegraphics[width=\linewidth]{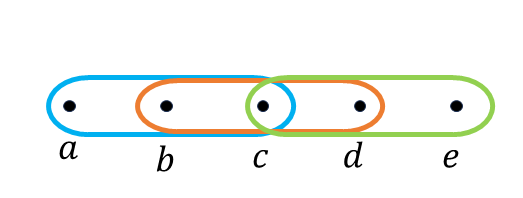}
\caption{A tight $P_3$, abbreviated as ${\rm TP}_3$.}
\label{fig:tightp3}
\end{minipage}\hfill
\begin{minipage}[b]{0.3\linewidth}
\centering
\includegraphics[width=1.1\linewidth]{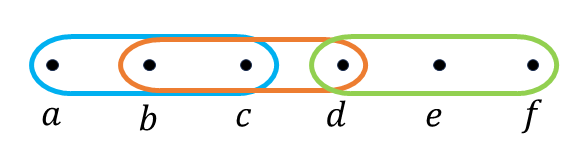}
\caption{A special $P_3$, abbreviated as ${\rm SP}_3$.}
\label{fig:tightloosepath}
\end{minipage}\hfill
\begin{minipage}[b]{0.3\linewidth}
\centering
\includegraphics[width=\linewidth]{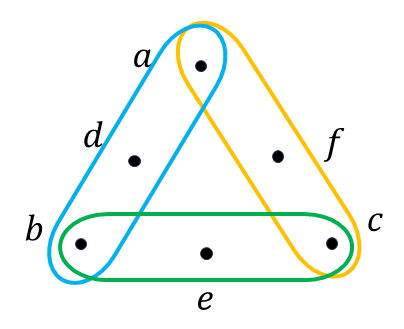}
\caption{A loose $C_3$, abbreviated as ${\rm LC}_3$.}
\label{fig:loosec3}
\end{minipage}
\end{figure}\begin{theorem}\label{3constantupperbound}
For any $3$-graph $H$, we have the following results,
$$
 C_3(n,H)\left\{
\begin{array}{cc}
 \le7    &  H\in \mathcal{H};\\
 =\Omega \left( \frac{\log^{(3)}(n)}{\log^{(4)}(n)} \right)^{\frac{1}{16}}& H\notin \mathcal{H},
\end{array} \right.
$$
where $\log^{(r)}(n) :=\underbrace{\log...\log }_{r} n$.
\end{theorem}

For some specific hypergraphs, such as special paths, cliques, sunflowers, and matchings, we can obtain better lower bounds, which are presented in Theorems \ref{STP3} -- \ref{cor}.

Let $P_t$ be a (2-graph) path consisting of $t$ edges. 
As we mentioned before, Alon and Ben-Eliezer \cite{zbMATH06056109} gave a lower bound for $C_2(n,P_3)$. They also showed that $C_2(n,P_7),\ C_2(n,P_8) =\Omega (n^{1/6})$ and $C_2(n, P_t) = \Omega(n^{1/4})$ for $t\geq 9$. Later, 
Cheng and Xu \cite{zbMATH07882976} proved that
$C_2(n, P_4) = \Omega(n^{1/5})$,
$C_2(n, P_t) = \Omega(n^{1/3})$ for $t \in \{5, 6, 7\}$, and
$C_2(n, P_t) = \Omega(n^{1/2})$ for $t \geq 8$. For $3$-graphs,
based on ${\rm SP_3}$, we study the following two special classes of 3-uniform paths with $t$ edges:
  \begin{description}
    \item[${\rm SP}_t^1$] --- a 3-graph on vertex set $\{v_1,...,v_{\lceil\frac{3t+3}{2}\rceil}\}$ and edge set 
    \[\big\{ (v_1,v_2,v_3),(v_3,v_4,v_5),(v_4,v_5,v_6),...,(v_{\lceil\frac{3t+3}{2}\rceil-2},v_{\lceil\frac{3t+3}{2}\rceil-1},v_{\lceil\frac{3t+3}{2}\rceil})  \big\};\]
    \item[${\rm SP}_t^2$] --- a 3-graph on vertex set $\{v_1,...,v_{\lceil\frac{3t+2}{2}\rceil}\}$ and edge set 
    \[ \big\{ (v_1,v_2,v_3), (v_2,v_3,v_4), (v_4,v_5,v_6),\cdots(v_{\lceil\frac{3t+2}{2}\rceil-2},v_{\lceil\frac{3t+2}{2}\rceil-1},v_{\lceil\frac{3t+2}{2}\rceil}) \big\}.
    \]
\end{description}
Clearly, $\rm{SP}_t^1$ and $\rm{SP}_t^2$ become identical when $t$ is odd.
We establish the following results.

\begin{theorem}\label{STP3}
The following holds:

\begin{itemize}
    \item $C_3(n,{\rm SP}_3)=\Omega((\frac{\log n}{\log \log n})^{\frac{1}{8}})$;
    
    \item $C_3(n, {\rm SP_4^1})=\Omega(n^{\frac{1}{8}})$ and  $C_3(n, {\rm SP_4^2})=\Omega(n^{\frac{1}{7}})$;

\item $C_3(n, {\rm SP^1_t})=\Omega(n^{\frac{1}{3}})$ and $C_3(n, {\rm SP^2_t})=\Omega(n^{\frac{1}{3}})$ for $t\ge 5$.

\end{itemize}
\end{theorem}

For complete graphs $K_t$, Cheng and Xu \cite{zbMATH07882976} established that $C_2(n, K_t) = \Omega(n^{2/3})$ for $t \geq 8$. 
Subsequently, Janzer and Janzer \cite{JANZER2024134} proved that $C_2(n, K_t) = \Omega\left(n^{1 - \frac{4}{t + 2}}\right)$ for even $t \geq 4$ and 
$C(n, K_t) = \Omega\left(n^{1 - \frac{10}{t - 3}}\right)$ for sufficiently large odd $t$.
We determine a polynomial lower bound for the local rainbow coloring number of large complete hypergraphs.

\begin{theorem}\label{thm clique}
$C_r(n,K_{p}^{(r)})=\Omega(n^{\frac{r}{r+1}})$ for $p\geq(2r-1)r+2(r-1)$ and $r\geq 2$.
\end{theorem}

Let $I_t$ be a graph with $t$ independent edges and $S_t$ be a star with $t$ edges.  Alon and Ben-Eliezer \cite{zbMATH06056109} showed that $C_2(n,I_4) = \Omega(n^{1/6})$, $C_2(n,I_t) =\Omega (n^{1/4})$ for $t\geq 5$,
 $C_2(n,S_4) =\Omega (n^{1/4})$ and $C_2(n,S_t) = \Omega(n^{1/3})$ for $t\geq 5$. Cheng and Xu \cite{zbMATH07882976} proved that
$C(n, I_4) = \Omega(n^{1/5})$,
$C(n, I_t) = \Omega(n^{1/3})$ for $t \in \{5, 6\}$,
$C(n, I_t) = \Omega(n^{1/2})$ for $t \geq 7$,
$C(n, S_4) = \Omega(n^{1/3})$, and
$C(n, S_t) = \Omega(n^{1/2})$ for $t \geq 5$. 
In hypergraphs, a star-like structure is often called a sunflower, and it is defined as follows. Let $r,d,m$ be nonnegative integers and let $D$ be a vertex set of size $d$. A \emph{sunflower $S_r(d,m)$} is an $r$-graph with edge set $E(S_r(d,m))=\{e_1,e_2,...,e_m\}$, where $e_i\cap e_j=D$ for every different $i,j\in[m]$. The set $D$ is referred to as the \emph{core} of this sunflower.

\begin{theorem}\label{sunflowers}
    For any integer $t\geq4$, $r\geq 3$ and $d\geq 0$, there exists a sufficiently small constant $c=c(t,r)$ such that $C_r(n,S_r(d,t))\geq\min\{cn^{\frac{1}{2(r-d)+1}},cn^{\frac{1}{d+1}}\}$.
\end{theorem}
Actually, the sunflower $S_r(0,t)$ is the hypergraph $M_t$ consisting of $t$ independent edges, i.e., a \textit{$t$-matching}. By Theorem $\ref{sunflowers}$, we get the following result.
\begin{corollary}\label{cor}
For any integer $t\geq4$ and $r\geq 3$, $C_r(n,M_t)=\Omega(n^{\frac{1}{2r+1}})$.
\end{corollary}

In addition, we can give a polynomial lower bound for the 3-graphs with at least 163 edges, which can be directly derived from the following result 
of Erd\H os and Rado \cite{Sunflower} and Theorem \ref{generalbound}. 

Let $g(r,m)$ be the minimum integer $N$ such that any $r$-graph with at least $N$ edges contains a sunflower $S_r(d,m)$ for some $0\leq d\leq r-1$. Erd\H os and Rado \cite{Sunflower} gave the following bounds:
\begin{equation*}
    (m-1)^r\leq g(r,m)\leq(m-1)^rr!+1.
\end{equation*}
By taking $m=4$, we know that any $r$-graph $H$ with at least $g(r,4)$ edges contains a sunflower $S_r(d,4)$ for some $0\leq d\leq r-1$ as a subgraph. 
Since $C_r(n,S_r(d,m))$ is polynomial in terms of $n$ by Theorem \ref{sunflowers}, we can get the following result by Theorem \ref{lemma polynomial subgraph imply  polynomial graph} which will be shown in Subsection \ref{generallowerbound}. 
\begin{theorem}\label{generalbound}
    For any $r$-graph $H$ with at least $g(r,4)$ edges there is a constant $b=b(H)>0$ so that $C_r(n,H)=\Omega(n^b)$.
\end{theorem}

For the upper bounds, we establish the following theorem for $r$-graphs.
 \begin{theorem}\label{general_bounds}
Given a $r$-graph $H$ consisting of $h$ vertices, it holds that
\[
C_r(n, H) = O\left( n^{\frac{h-r}{h}} \cdot h^{2r + \frac{r}{h}} \right).
\]
\end{theorem}

\noindent \textbf{Organization of the paper.}
The subsequent sections of this paper proceed as follows. 
Section \ref{upperbounds} presents results concerning upper bounds for the local rainbow coloring number. 
Specifically, Subsection \ref{subsection:up-1} contains the proof of Theorem \ref{general_bounds}. 
Subsection \ref{subsection:up-2} establishes Theorems \ref{constantnotlarge} and \ref{3constantupperbound}, while Subsection \ref{subsection:up-3} provides subpolynomial upper bounds for a specific hypergraph ${\rm TC_e}$ (illustrated in Figure \ref{fig:c3}). 
Section \ref{lower_bound} examines lower bounds for the local rainbow coloring number across various graph classes. 
 Subsection \ref{generallowerbound} presents Theorem \ref{lemma polynomial subgraph imply polynomial graph} which show a relationship between the local rainbow coloring number of a graph and that of its subgraphs.
 Subsequent subsections focus on specialized structures: Subsection \ref{sec:cliques} proves Theorem \ref{thm clique} for complete hypergraphs, Subsection \ref{sec:sunflowers} demonstrates Theorem \ref{sunflowers} by establishing polynomial bounds for sunflowers with at least four petals, and finally, Subsection \ref{sec:SP3} concludes with the proof of Theorem \ref{STP3}.

\section{Upper bounds}\label{upperbounds}

In this section, we establish the upper bound of the local rainbow coloring number of general $r$-graphs and some specific hypergraphs. 
%In Subsection \ref{subsection:up-1}, we prove Theorem \ref{general_bounds}. 
%In Subsection \ref{subsection:up-2}, we proceed to the proofs of Theorems \ref{constantnotlarge} and \ref{3constantupperbound}. 
%In Subsection \ref{subsection:up-3}, we demonstrate that the unboundedness of $C_3(n, H)$ does not necessarily imply a polynomial growth rate by presenting an example of $H$, each exhibiting a subpolynomial upper bound for the local rainbow coloring number. 

\subsection{General bounds} \label{subsection:up-1}
To prove Theorem \ref{general_bounds}, we use probabilistic method and Lov\'{a}sz local lemma.

\begin{lemma}[\cite{alon2016probabilistic} Lov\'{a}sz local lemma]\label{local lem}
Let $A_1, A_2, \cdots, A_n$ be events, with $Pr[A_i]\leq p$ for all i. Suppose that each $A_i$ is independent of a set of all other $A_j$ except for at most $d$ of them. If $$ep(d+1)\leq 1,$$ then with positive probability, none of the events $A_i$ occur.
\end{lemma}

\begin{proof}[Proof of \autoref{general_bounds}]
It is sufficient to prove the theorem for the specific case that $H$ is a complete $r$-graph with $h$ vertices. This is due to the fact that for any $r$-graph $H$ with $h$ vertices, if a collection of colorings is an $(n,r,K_{h}^{(r)})$-local coloring, then it will also ensure an $(n,r,H)$-local coloring.

Let $k = h^{2r+\frac{r}{h}} n^{\frac{h-r}{h}}$. For every $v \in V(K_n^{(r)})$ and $e \in E(K_n^{(r)})$, let $f_v(e)$ be an random variable that follows a uniform distribution among all $k$ possible colors. Note that for $u\neq v$ or $e\neq e'$, $f_v(e)$ is independent of all other events $f_u(e')$. For any $H^{*}$ isomorphic to $H$, we use $A(H^{*})$ to denote the event that the collection of colorings $f_v$ is not an $(n,r,H^{*})$-local coloring collection. Then we will show that none of the events $A(H^{*})$ occur with positive probability. 

For a fixed vertex $v \in V(H^{*})$, the probability that there exist two edges in $E(H^{*})$ receive the same color in $f_v$ can be bounded by 
\[k\cdot\frac{1}{k^2}\cdot \binom{\binom{h}{r}}{2}\leq \frac{h^{2r}}{2k}.\]
Since the colorings are chosen independently, we get 
\[\mathbb{P}[A(H^{*})]\leq \left(\frac{h^{2r}}{2k}\right)^h. \]

For two distinct copies $H_1^{*}$, $H_2^{*}$ of $H$, if $E(H_1^{*})\cap E(H_2^{*})=\emptyset$, then the events $A(H_1^{*})$ and $A(H_2^{*})$ are independent. Thus, each event $A(H_1^{*})$ is mutually independent of all other events $A(H_2^{*})$ besides those for which $|E(H_1^{*})\cap E(H_2^{*})|\geq 1$. Consequently, for every $H^{*}$, $A(H^{*})$ is independent of all but at most $\binom{h}{r}\binom{n-r}{h-r}< h^rn^{h-r}$ other events. By the Theorem \ref{local lem}, we get a positive probability that none of the events $A(H^{*})$ occur, it follows from $$e\cdot h^rn^{h-r}\cdot \left(\frac{h^{2r}}{2k}\right)^h < 1.$$ Hence, there exist $n$ colorings that satisfy the requirements.
\end{proof}

\subsection{Constant bounds} \label{subsection:up-2}

In this section, we prove the following theorem which gives a necessary and sufficient condition for the local rainbow coloring number to be a constant. 

\begin{theorem}\label{constantnotlarge}
For an $r$-graph $H$, $C_r(n,H)$ is bounded by a constant $C(H)$ if and only if $H$ is not $2$-locally large.
\end{theorem}

We generalize the notion of 2-locally large graphs introduced by Alon and Ben-Eliezer \cite{alon2016probabilistic} to hypergraphs.
Let $H$ be an $r$-graph on $n$ vertices and let $\sigma:V(H)\rightarrow [n]$ be a permutation. 
For each vertex $x \in V(H)$, we define $2r+1$ disjoint sets of edges, denoted as $T_x^i$ for $1\leq i \leq 2r+1$, such that $\bigcup_{i=1}^{2r+1}T_x^{i}=E(H)$. 
The first $r$ sets, $T_x^i$ for $1\leq i \leq r$, consist of all edges $e$ that include $x$, where $\sigma(x)$ is the $i$-th largest element within $e$.
The remaining $r+1$ sets, $T_x^i$ for $r+1 \leq i \leq 2r+1$, consist of all edges $e$ that do not contain $x$, where $\sigma(x)$ would be the $(i-r)$-th largest element if $x$ is added to $e$. Formally,
\[ 
T_x^i = 
\begin{cases} 
\begin{aligned}
&\{e \,|\, e = \{v_1,\ldots,v_{i-1},x,v_{i+1},\ldots,v_r\}, \\
&\phantom{\{e = \{ }  \sigma(v_1) < \cdots < \sigma(v_{i-1}) < \sigma(x) < \sigma(v_{i+1}) < \cdots < \sigma(v_r)\}
\end{aligned}
& \text{if } 1 \leq i \leq r, \\
\begin{aligned}
&\{e \,|\, e = \{v_1,v_2,\ldots,v_r\} \text{ and } x \notin e, \\
&\phantom{\{e = \{ }  \sigma(v_1) < \cdots < \sigma(v_{i-r-1}) < \sigma(x) < \sigma(v_{i-r}) < \cdots < \sigma(v_r)\}
\end{aligned}
& \text{if } r+1 \leq i \leq 2r+1.
\end{cases}
\]

Now, we are ready to give a formal definition of 2-locally large property.

\begin{definition} \label{def:2-locally large}
    An $r$-graph $H$  with $n$ vertices is \emph{2-locally large} if there is a bijection 
    $\sigma:V(H)\rightarrow [n]$ such that for every vertex $x\in V(H)$, there exists one of the sets $T_x^i$, where $i \in [2r+1]$, that  contains at least two elements.
\end{definition}

We prove a stronger theorem, which implies Theorem \ref{constantnotlarge} and provides additional information beyond it.

\begin{theorem}\label{stronger}
Given an $r$-graph $H$ with $h$ vertices, we have
\[
C_r(n,H)
\begin{cases} 
\leq 2r+1 & \text{if } H \text{ is not 2-locally large,} \\
\geq c \left( \frac{\log^{(r)}(n)}{\log^{(r+1)}(n)} \right)^{\frac{1}{(r+1)^2}} & \text{if } H \text{ is 2-locally large,}
\end{cases}
\]

for some constant $c=c(r,h)$.
    
\end{theorem}

The proof of Theorem \ref{stronger} requires some tools.
We begin by giving the following two propositions.
\begin{proposition}\label{3edges-2-locally large}
    If an $r$-graph $H$ is $2$-locally large, then $H$ has at least three edges.
\end{proposition}

\begin{proof}
  If $H$ has at most two edges, then there is a vertex $x$ contained in exactly one edge. Now, for any permutation $\sigma$ on the vertices of $H$, we have that $T_x^i$ contains at most one element for every $i\in [2r+1]$.
\end{proof}

\begin{proposition}\label{property 2-locally large}
    For every $r$-graph $H$, if $H$ contains a $2$-locally large subgraph, then $H$ is $2$-locally large.
\end{proposition}

\begin{proof}
Let $H'$ be a 2-locally large subgraph of $H$ on $h'$ vertices. By Definition \ref{def:2-locally large}, there exists a permutation $\sigma': V(H') \rightarrow [h']$ such that for every vertex $x \in V(H')$, there exists one of the sets $T_x^i$ containing at least two edges, where $i \in [2r+1]$.
Define $\sigma: V(H) \rightarrow [h]$, where $h$ is the number of vertices of $H$, such that $\sigma(x) = \sigma'(x)$ for all $x \in V(H')$ and $\sigma(x)$ is an arbitrary assignment for $x \in V(H) \setminus V(H')$ that ensures $\sigma$ is a permutation on $V(H)$.
Then, we need to verify that for every vertex $x \in V(H)$, there exists at least one set $T_x^i$ that contains at least two edges in $H$.
For $x \in V(H')$, this is true by the 2-locally large property of $H'$.
For $x \in V(H) \setminus V(H')$, $T_x^{2r+1}$ contains at least three edges in $H$ because $H'$ has at least three edges by Proposition \ref{3edges-2-locally large}. 
\end{proof}

Then we introduce the following result about the multicolor Ramsey number for hypergraphs, which plays a crucial role in the proof of Theorem \ref{stronger}.

\begin{theorem} [\cite{https://doi.org/10.1112/plms/s2-30.1.264}]\label{ramsey}
For all integers $k\geq t\geq 2,q\geq 2$, there exists some $N$ such that the following holds. In any $q$-coloring $\chi$: $E(K_N^{(t)})\to 
[q]$, there is a monochromatic
 copy of $K_k^{(t)}$. In other words, there exist $k$ vertices such that each of the $\binom{k}{t}$ $t$-tuples among
 them receive the same color under $\chi$. 
\end{theorem}
Define the \emph{$t$-uniform Ramsey number $r_t(k;q)$} to be the least $N$ for which Theorem \ref{ramsey} is true.
The following Lemma \ref{lemma multicolor hypergraph ramsey} is a corollary of the result of Erd\H{o}s and Rado \cite{Multicolorranmseynumber}.
\begin{lemma}[\cite{Multicolorranmseynumber}]\label{lemma multicolor hypergraph ramsey}
\begin{equation*}
    r_t(k;q)\leqslant2^{2^{.^{.^{.^{2^{(C_tq\log q)k}}}}}}\Big\}~t-1 ~tower,
\end{equation*}
where $C_t$ is a constant depending only on $t$.
\end{lemma}

Now, we are ready to complete the proof of \autoref{stronger}.
\begin{proof}[Proof of \autoref{stronger}]
    Let $H$ be an $r$-graph on $m$ vertices. 
    First we show that if $H$ is not 2-locally large, then there is an $(n,r,H)$-local coloring of $K_n^{(r)}$ using at most $2r+1$ colorings, and therefore $C_r(n,H)\leq 2r+1$. For every vertex $v\in V(K_n^{(r)})$, we set $f_v(e)=i$ if and only if $e\in T_v^i$. Let $H'$ be a copy of $H$ in $K_n^{(r)}$. Since $H'$ is not 2-locally large, for any permutation of $V(H')$, there is a vertex $x\in V(H')$ such that $|T_x^i\cap E(H')|\leq 1$ for every $i\in [2 r+1]$. Therefore, for any distinct edges $e_1,e_2\in E(H')$, we have $f_x(e_1)\neq f_x(e_2)$.
    
    We now prove the lower bound when $H$ is 2-locally large. Let $\sigma_0:V(K_n^{(r)})\rightarrow [n]$ be the bijection given by Definition \ref{def:2-locally large} and let $K_n^{(r+1)}$ be the complete $(r+1)$-graph on the same vertex set of $K_n^{(r)}$. 
    For each $v\in V(K_n^{(r)})$, let $f_v$ be a $k$-edge-coloring  of $K_n^{(r)}$.
    Based on $\{f_v: v\in V(K_n^{(r)})\}$, we define a $k^{(r+1)^2}$-edge-coloring $g$ of $K_n^{(r+1)}$ in the following way.
    
    For every edge $e = \{u_1, u_2, \ldots, u_{r+1}\}$ in $K_n^{(r+1)}$ where the vertices are ordered such that $\sigma_0(u_1) < \sigma_0(u_2) < \ldots < \sigma_0(u_{r+1})$, we define the coloring of $e$ as an ordered $(r+1)^2$-tuple:
\begin{align*}
g(e)=
\left(
\begin{gathered}
f_{u_1}(e \backslash \{u_{r+1}\}), f_{u_1}(e \backslash \{u_r\}), \ldots, f_{u_1}(e \backslash \{u_1\}), \\
f_{u_2}(e \backslash \{u_{r+1}\}), f_{u_2}(e \backslash \{u_r\}), \ldots, f_{u_2}(e \backslash \{u_1\}), \\
\vdots \\
f_{u_{r+1}}(e \backslash \{u_{r+1}\}), f_{u_{r+1}}(e \backslash \{u_r\}), \ldots, f_{u_{r+1}}(e \backslash \{u_1\})
\end{gathered}
\right).
\end{align*}   
Choose a sufficiently large value for  $s$. 
By Lemma \ref{lemma multicolor hypergraph ramsey}, if $n\geq r_{r+1}(s;k^{(r+1)^2})$, then there exists a monochromatic $K_s^{(r+1)}$. 

Let $y$ be the vertex in $K_s^{(r+1)}$ with the maximum labeling according to $\sigma_0$. Consider a copy $H'$ of $H$ contained in $K_s^{(r+1)}$ such that $y \notin V(H')$ (we can add this additional condition because $s$ is sufficiently large).
Since $H$ is 2-locally large, for each vertex $x\in V(H')$, there exists some $i\in [2 r+1]$ such that $|T_x^i\cap E(H')|\geq 2$. Take two such edges $e_1$ and $e_2$ from $T_x^i\cap E(H')$. 
If $x\not\in e_1$, then $x\not\in e_2$ since $e_1,e_2\in T_x^i$. Now, $e'_1:=\{x\} \cup e_1$ and $e'_2:=\{x\} \cup e_2$ are two edges of that monochromatic $K_s^{(r+1)}$. Note that, after sorting the vertices in $e'_1$ and $e'_2$ according to $\sigma_0$, the position of $x$ is the same in both sets. This implies $f_x(e_1) = f_x(e_2) $ according to the definition of the coloring $g$.
On the other hand, if $x \in e_1$, then $x \in e_2$ as well, since both $e_1$ and $e_2$ belong to $T_x^i$. 
In this case, recall that $y$ is the vertex in $K_s^{(r+1)}$ possessing the maximum labeling according to $\sigma_0$ and is excluded from both $e_1$ and $e_2$. After arranging the vertices in $e''_1 := \{y\} \cup e_1$ and $e''_2 := \{y\} \cup e_2$ based on $\sigma_0$, $y$ takes up the final position in both sets simultaneously. Consequently, $x$ also holds the same position in both sets. Therefore, according to the definition of $g$, we have $f_x(e''_1 \setminus \{y\}) = f_x(e''_2 \setminus \{y\})$. In either case, we finally have $f_x(e_1)=f_x(e_2)$ and thus $K_n^{(r)}$ has no $(n,r,H)$-local coloring by $s$ colors.

Therefore, if $H$ is 2-locally large, then $C_r(n,H)\geq s$.
Since $n\geq r_{r+1}(s;k^{(r+1)^2})$, we have $s\geq c (\frac{\log^{(r)}(n)}{\log^{(r+1)}(n)})^\frac{1}{(r+1)^2}$ for some constant $c=c(r,h)$, as required.
\end{proof}

Applying Theorem \ref{stronger}, we can give a proof of Theorem \ref{3constantupperbound}, 
by distinguishing between all $3$-graphs, based on whether or not they are 2-locally large.
\begin{proof}[Proof of Theorem~\ref{3constantupperbound}]
By Theorem \ref{stronger}, if $H$ is not $2$-locally large, then $C_3(n,H)\leq 7$. 
Combining with Propositions \ref{3edges-2-locally large} and \ref{property 2-locally large}, we only need to show the following statements.
\begin{enumerate} 
\item All 3-graphs with three edges are not 2-locally large unless $H$ is one of ${\rm TP}_3$, ${\rm SP}_3$ and ${\rm LC}_3$, possibly with some isolated vertices added.
\item All 3-graphs with exactly four edges are 2-locally large.
\end{enumerate}

First, we show that $\mathrm{TP}_3$, $\mathrm{SP}_3$ and $\mathrm{LC}_3$ are $2$-locally large. Specifically, for $\mathrm{TP}_3$, we provide a permutation $\sigma$ of the vertices such that $\sigma(c) < \sigma(d) < \sigma(a) < \sigma(b) < \sigma(e)$. For $\mathrm{SP}_3$, we define a permutation $\sigma$ such that $\sigma(f) < \sigma(e) < \sigma(c) < \sigma(b) < \sigma(d) < \sigma(a)$. For $\mathrm{LC}_3$, we offer a permutation $\sigma$ such that $\sigma(a) < \sigma(e) < \sigma(d) < \sigma(c) < \sigma(b) < \sigma(f)$.
It can be checked that each of those permutations satisfies the requirement of Definition \ref{def:2-locally large}.
For other 3-graphs containing three edges that are not 2-locally large, they are verified by the computer program detailed in Appendix \ref{fulu2}. Similarly, Appendix \ref{fulu1} provides a computer program that confirms item 2.
\end{proof}

\noindent \textbf{Remark:} It is not difficult to verify item 1 and item 2 above without computers, but we have to deal with numerous cases. So we use computers to assist with proof. Here we present an example of purely theoretical proof of $M_3$, a 3-graph with three disjoint edges $e_1,e_2,e_3$, where $e_i=\{x_i,y_i,z_i\}$ for $i\in [3]$. We need to show that for any permutation $\sigma:V(M_3)\rightarrow[9]$, there exists some vertex $x\in V(M_3)$ such that $|T_x^i|\leq 1$ for all $i\in [7]$, and we call such vertex $x$ \emph{bad}. Without loss of generality, we assume $\sigma(x_1)=1$. Then $\sigma(y_1)=2$ or $\sigma(z_1)=2$, otherwise the vertex $v$ with $\sigma(v)=2$ is bad. Indeed, we can assume $v\in e_2$, then $e_1\in T_v^5$, $e_2\in T_v^1$ and $e_3\in T_v^4$. Let $\sigma(y_1)=2$ (it does not affect the final result if let $\sigma(z_1)=2$), then $\sigma(z_1)=3$ for the similar reason. But the vertex $v$ with $\sigma(v)=4$ must be bad.  

\subsection{Subpolynomial bounds} \label{subsection:up-3}

The upper bounds of local rainbow colorings for specific graphs $H$, as seen in \cite{JANZER2024134}, are linked to the Erd\H{o}s-Gy\'arf\'as function, a concept from Ramsey Theory introduced by Erd\H{o}s and Shelah \cite{10.5555/21936.25441}. 
The precise values of the Erd\H{o}s-Gy\'arf\'as function for numerous cases remain open, and determining this value is an ongoing research endeavor. More results about the Erd\H{o}s-Gy\'arf\'as function can be seen  \cite{Conlon201549,10.1093/imrn/rnu190,https://doi.org/10.1112/plms/pdu049,Eichhorn2000441,Erdos1997459,https://doi.org/10.1002/jgt.21876,TYJOUR}.

\begin{definition}[\cite{10.5555/21936.25441}]
    Let $p, q, r, n \geq 2$ be positive integers with $q \leq \binom{p}{r}$. An edge-coloring of the $r$-graph $K^{(r)}_n$ is a $(p, q)$-coloring if at least $q$ distinct colors appear among any $p$ vertices. Let $f_r(n, p, q)$ be the smallest positive integer $k$ such that there exists a $k$-edge-coloring of $K^{(r)}_n$ forming a $(p, q)$-coloring.
\end{definition}

Janzer and Janzer \cite{JANZER2024134} gave the following theorem, which will be used to construct the upper bounds of ${\rm TC_e}$.

\begin{theorem}[\cite{JANZER2024134}]\label{r1r}
 For any $r\geq 3$, we have \[f_r(n,r+1,r)\leq e^{(\log n)^{2/5+o(1)}}=n^{o(1)}.\]
\end{theorem}

Based on the theorem presented above, it is evident that the local rainbow coloring number of $\mathrm{TP}_3$ is bounded by a subpolynomial function. Nonetheless, $\mathrm{TP}_3$ is not unique in this regard. We proceed to introduce another 3-graph that is 2-locally large and possesses a subpolynomial upper bound on its local rainbow coloring number using a similar construction of colorings. It is a tight $C_3$ with a pendent edge,  defined by
\begin{description}
    \item[${\rm TC}_e$] --- a 3-graph on vertex set $\{a,b,c,d,e\}$ and edge set $\big\{ \{a,b,d\}, \{b,c,d\}, \{c,d,e\}, \{a,c,d\}  \big\}.$% (see \autoref{fig:c3}).
\end{description}

\begin{figure}[H]
   \centering
   \includegraphics[width=0.3\linewidth]{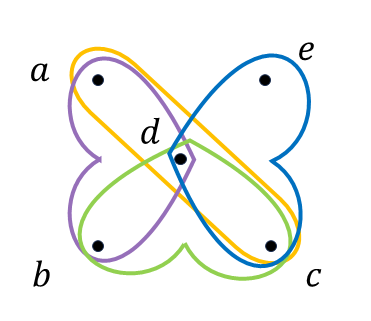}
   \caption{A tight $C_3$ with a pendant edge, abbreviated as ${\rm TC}_e$.}
   \label{fig:c3}
\end{figure}

\begin{theorem}
 $C_3(n,{\rm TC}_e)=n^{o(1)}$.  
\end{theorem}
\begin{proof}
By \autoref{r1r}, there is an edge-coloring $\rho$ of $K_n^{(4)}$ using $e^{(\log n)^{2/5+o(1)}}$ colors such that among any 5 vertices at least 4 colors appear, and an edge-coloring $\gamma$ of $K_n^{(3)}$ using another $e^{(\log n)^{2/5+o(1)}}$ colors such that among any 4 vertices at least 3 colors appear. 
We take a similar coloring rule as above, i.e.
 \[  f_v(e)=\begin{cases}
  \gamma(e) &\text{if $v\in e$,}\\
  \rho(e\cup\{v\}) &\text{if $v\notin e$}.
\end{cases}  \]

Let $abcde$ be any copy of ${\rm TC_e}$ (see \autoref{fig:c3}).
Then we claim that this copy must be rainbow in one of $f_a$, $f_b$ or $f_e$. Indeed, if this copy of ${\rm TC_e}$ is not rainbow in $f_a$, then $\gamma(abd)=\gamma(acd)$ or $\rho(abcd)=\rho(acde)$. If this copy of ${\rm TC_e}$ is not rainbow in $f_b$, then $\gamma(abd)=\gamma(bcd)$ or $\rho(abcd)=\rho(bcde)$. By the definition of $\gamma$ and $\rho$, we know that $\gamma(abd)=\gamma(acd)$ and $\gamma(abd)=\gamma(bcd)$ cannot hold together, nor can $\rho(abcd)=\rho(acde)$ and $\rho(abcd)=\rho(bcde)$. Without loss of generality, we assume $\gamma(abd)=\gamma(acd)$ and $\rho(abcd)=\rho(bcde)$. Now, we see the coloring of $abcde$ in $f_e$, where four edges be colored by $\gamma(cde), \rho(abde), \rho(bcde),\rho(acde)$, is rainbow. 
Hence $C_3(n,{\rm TC_e})\leq e^{(\log n)^{2/5+o(1)}}+e^{(\log n)^{2/5+o(1)}}=n^{o(1)}$.
\end{proof}

\section{Lower bounds}\label{lower_bound}

In this section, we focus %more 
on the order of magnitude of the lower bound for $C_r(n,H)$.
%To this end, in Subsection \ref{generallowerbound}, we show that if $C_r(n,H')$ is polynomial and $H'$ is a subgraph of $H$, then $C_r(n,H)$ is polynomial.
%In Subsections \ref{sec:cliques}, we provide polynomial lower bounds for the local rainbow coloring numbers of classical standard types of hypergraphs including cliques. 
%In Subsection \ref{sec:sunflowers}, we prove \autoref{sunflowers} which shows that the lower bound of $C_r(n,S_r(d,4))$ is polynomial for all $0\leq d\leq r-1$.
%Finally, we show in Subsection \ref{sec:SP3} that $C_3(n,{\rm SP}_3)$ increases at a rate comparable to $(\frac{\log n}{\log \log n})^{\frac{1}{8}}$. This finding is also pertinent to Problem \ref{originalproblem}.

\subsection{Polynomial bounds}\label{generallowerbound}
The following theorem stems from the work of Janzer and Janzer \cite{JANZER2024134}, in which the core idea was initially introduced. 
To ensure the completeness of our exposition, we present a detailed proof here. 
It shows that if $C_r(n,H')$ is polynomial and $H'$ is a subgraph of $H$, then $C_r(n,H)$ is polynomial.

\begin{theorem}\label{lemma polynomial subgraph imply  polynomial graph}

If $H$ is an $r$-graph and $T$ is a subgraph of $H$, then $C_r(n,H)\geq C_r(n,T)^{\frac{1}{|V(H)|+1}}$.
\end{theorem}

\begin{proof}
To be convenient, let $|V(H)|=h$, $|V(T)|=t$ and $h-t=c>0$. Let $T'$ be the hypergraph obtained from $T$ by adding $c$ distinct isolated vertices, which implies that $T'$ is a spanning subgraph of $H$. By the definition of local rainbow coloring number, we have $C_r(n,T')\leq C_r(n,H)$. Let $k=C_r(n,T')$.
So there are $n$ $k$-edge-colorings $f_v$ of $K_n^{(r)}$ for each $v\in V(K_n^{(r)})$ which form an $(n,r,T')$-local coloring. Let $u_1,\dots,u_{h}$ be arbitrary $h$ distinct vertices in $V(K_n^{(r)})$. For each $v\in V(K_n^{(r)})$, define an edge-coloring $f_v'$ of $E(K_n^{(r)})$ by setting
\[f'_v(e)=(f_v(e),f_{u_1}(e),\dots,f_{u_h}(e)).\]
This is an edge coloring using at most $k^{h+1}$ colors. 

We claim that the coloring arrangement $f'_v$ forms an $(n,r,T)$-local coloring. Otherwise, there exists a copy $F$ of $T$ in $K_n^{(r)}$ such that for every $v\in V(F)$, $F$ contains two edges with the same color in $f'_v$. By definition, any such pair of edges have the same color in all of $f_v,f_{u_1},...,f_{u_h}$. Since $h>t$, we can find a copy $F'$ of $T'$ (obtained by adding $c$ distinct vertices $u_i$ to $F'$) such that for each $v\in V(F')$, $F'$ contains two edges with the same color in $f_v$. This is a contradiction. So we have $C_r(n,T)\le (C_r(n,T'))^{h+1}\le (C_r(n,H))^{h+1}$.
\end{proof}

\subsection{Cliques} \label{sec:cliques}
In this subsection, we provide a polynomial lower bound for the local rainbow coloring numbers of cliques. 

\begin{proof}[Proof of Theorem \ref{thm clique}]
Let $K$ be an abbreviation of $K_{2n}^{(r)}$.
Indeed, we will prove $C_r(2n,K_{p}^{(r)})=\Omega(n^{\frac{r}{r+1}})$ which implies $C_r(n,K_{p}^{(r)})=\Omega(n^{\frac{r}{r+1}})$.
For a contradiction, we assume that $C_r(2n,K_{p}^{(r)})\leq k=c_1n^{\frac{r}{r+1}}$, where $c_1$ is a sufficiently small constant, which means there exists a set of $2n$ $k$-edge-colorings that is an $(2n,r,K_{p}^{(r)})$-local coloring. Let $V(K)=A\cup B$ be a balanced partition, i.e., $|A|=|B|=n$. Define $B^{[r-1]}$ as the set of all $(r-1)$-subsets of $B$. The elements of $B^{[r-1]}$ are sets of the form $\hat{b}=\{b_1,b_2,...,b_{r-1}\}$.
\begin{claim}\label{lemma clique}
Let $k\leq n$ and $c_r=\frac{1}{3r^{2r-2}}$. Now, the number of triples $(a,\hat{b}_1,\hat{b}_2)\in A\times B^{[r-1]}\times B^{[r-1]}$ with $f_a(a\hat{b}_1)=f_a(a\hat{b}_2)$ is at least $c_r\frac{n^{2r-1}}{k}$, where $a\hat{b}_i$ is the abbreviation for $a\cup\hat{b}_i$ for $i=1,2$.
\end{claim}
\begin{proof}
For each vertex $a\in A$, let $B^{[r-1]}_a(i)$ be a family of all $(r-1)$-sets $\hat{b}\in B^{[r-1]}$ satisfying $f_a(a\hat{b})=i$. The number of such triples $(a,\hat{b}_1,\hat{b}_2)$ with $f_a(a\hat{b}_1)=f_a(a\hat{b}_2)$ is at least
    \begin{equation*}
     \sum_{a\in A}\sum_{i=1}^k\binom{|B^{[r-1]}_a(i)|}{2}\geq n\cdot k\binom{\frac{|B^{[r-1]}|}{k}}{2}\geq n\cdot\frac{\binom{n}{r-1}^2}{3k}\geq \frac{n^{2r-1}}{3kr^{2r-2}}=c_r\frac{n^{2r-1}}{k},
    \end{equation*}
 where the first inequality holds by  Jensen's inequality.
\end{proof}

By the pigeonhole principle, there exist two different $(r-1)$-sets $\hat{b}_1,\hat{b}_2$ in $B^{[r-1]}$ such that there are at least 
  \[c_r\frac{n^{2r-1}}{k}\cdot\binom{|B^{[r-1]}|}{2}^{-1}\geq\frac{c_r}{c_1}n^\frac{1}{r+1}\] vertices $a\in A$ satisfying $f_a(a\hat{b}_1)=f_a(a\hat{b}_2)$. There exists a subset $A'\subseteq A$ of size $\frac{c_r}{c_1}n^\frac{1}{r+1}$ satisfying $f_a(a\hat{b}_1)=f_a(a\hat{b}_2)$ for each $a\in A'$.
Then the number of edges of $K[A']$ is $\binom{|A'|}{r}\geq 1000rk$ since $c_1$ is sufficiently small. For each vertex $b\in \hat{b}_1\cup\hat{b}_2$, we will partition $E(K[A'])$ into many disjoint $3$-sets $\{E_b^i\}$ and a remaining set $R_b=E(K[A'])\backslash(\cup_iE_b^i)$ as the following rules. For any $2k+1$ edges, we can find three edges with the same color under $f_b$ by the pigeonhole principle. Suppose we have got $\{E_b^i\}_{i=1}^j$. 
If $|E(K[A'])\backslash(\cup_{i=1}^jE_b^i)|\geq 2k+1$, then we arbitrarily choose three edges with the same color under $f_b$ to form $E_b^{j+1}$, otherwise the process ends and let $R_b=E(K[A'])\backslash(\cup_{i=1}^jE_b^i)$. Each $R_b$ is only a small piece of $E(K[A'])$, so does $\cup_{b\in \hat{b}_1\cup\hat{b}_2}R_b$. Let $E_b=\cup_i E_b^i$. We can choose $e_0\in\cap_{b\in \hat{b}_1\cup\hat{b}_2} E_b$ and $e_b\in E_b$ for each $b\in\hat{b}_1\cup\hat{b}_2$ such that $e_0$ and $e_b$ are two different edges in some $E_b^i$. Let $X=V(\hat{b}_1\cup\hat{b}_2\cup e_0\cup(\cup_b e_b))$. If $|X|<p$, then we arbitrarily add several different vertices of $A'$ to $X$. So we get a set $X$ of size $p$. For each $x\in X$, if $x\in A'$, then $f_x(x\hat{b}_1)=f_x(x\hat{b}_2)$; otherwise $x\in \hat{b}_1\cup\hat{b}_2$ and $f_x(e_0)=f_x(e_x)$. So for any colorings, there is a copy $K[X]$ of $K_{p}^{(r)}$ which is not rainbow under the colorings corresponding to all its vertices, which contradicts our hypothesis.
\end{proof}

\subsection{Sunflowers}\label{sec:sunflowers}

We prove \autoref{sunflowers} which shows that the lower bound of $C_r(n,S_r(d,4))$ is polynomial for all $0\leq d\leq r-1$.

\begin{proof}[Proof of \autoref{sunflowers}]
Let $k=\min\{cn^{\frac{1}{2(r-d)+1}}, cn^{\frac{1}{d+1}}\}$. For any collection of $n$ $k$-edge-colorings $f_v$ of $V(K_n^{(r)})$ with $v\in V(K_n^{(r)})$ and any vertex set $D\subseteq V(K_n^{(r)})$ with $|D|=d<r$. Let $\mathcal{E}$ be a set of all edges in $K_n^{(r)}$ such that $e_1 \cap e_2= D$ and $f_x(e_1)=f_x(e_2)$ for any two different edges $e_1,e_2\in \mathcal{E}$ and any vertex $x\in D$. The size of $\mathcal{E}$ is
\[|\mathcal{E}|\geq \frac{\lfloor\frac{n-d}{r-d}\rfloor}{k^d}\geq \frac{n-r}{(r-d)k^d},\]
where the first inequality holds because we can partition $V(K_n^{(r)})\backslash D$ into many $(r-d)$-sets (which can become $r$-set by adding all vertices in $D$) and a remaining set. The number of such $(r-d)$-set is $\lfloor\frac{n-d}{r-d}\rfloor$, and we using the pigeonhole principle to get the desired $\mathcal{E}$.

\begin{claim}\label{claimsunflower}
Let $\mathcal{E}'\subseteq \mathcal{E}$ with $|\mathcal{E}'|=\lceil \frac{n-r}{(r-d)k^d} \rceil$ and $P=V(K_n^{(r)})\setminus V(\mathcal{E}')$. We have that the number of triples $(e_1,e_2,p)\in \mathcal{E}' \times \mathcal{E}' \times P$ such that $f_p(e_1)=f_p(e_2)$ is at least $ \frac{n^3}{6(r-d)^2k^{2k+1}}$.
\end{claim}

\begin{proof}%[Proof of Claim \ref{claimsunflower}]
   Since $P=V(K_n^{(r)})\backslash  V(\mathcal{E}')$, we obtain that $|P|\geq n-(d+(r-d)|\mathcal{E}'|)\geq\frac{n}{2}$ for sufficiently 
 large $n$. For any $p\in P$ and $i\in [k]$, let $\mathcal{E}_p'(i)$ be a set of edges $e\in \mathcal{E}'$ such that  $f_p(e)=i$. 
The number of triples $(e_1,e_2,p)\in \mathcal{E}'\times \mathcal{E}'\times P$ such that $f_p(e_1)=f_p(e_2)$ is equal to
    \begin{equation*}\label{eq sunflower}
        \sum_{p\in P}\sum_{i\in[k]}\binom{|\mathcal{E}_p'(i)|}{2}\geq |P|\cdot k\cdot\binom{\frac{|\mathcal{E}'|}{k}}{2}\geq \frac{n}{2}\cdot\frac{n^2}{3(r-d)^2 k^{2d+1}}=\frac{n^3}{6(r-d)^2 k^{2d+1}},
    \end{equation*}
    where the first inequality holds by  Jensen's inequality and the second inequality holds by $k\leq c n^{\frac{1}{d+1}}.$
\end{proof}

Using the pigeonhole principle, there exist two different edges $e,e'\in \mathcal{E}'$ such that there are at least $\frac{n^3}{6(r-d)^2 k^{2d+1}}\cdot\frac{1}{|\mathcal{E}'|^2}\geq n^{1-\frac{1}{2(r-d)+1}}$ vertices $p\in P$ satisfying $f_p(e)=f_p(e')$, where we use $k\leq cn^{\frac{1}{2(r-d)+1}}$ for a small constant $c$. Let $X$ be a set of such vertices $p\in P$. The size of $X$ is at least $m:=\lceil n^{1-\frac{1}{2(r-d)+1}} \rceil$. Let $\mathcal{Q}:=\{Q_1,Q_2,\dots,Q_{\lfloor \frac{m}{r-d}\rfloor}\}$ be a family of disjoint $(r-d)$-subsets of $X$. Note that $Q_i\cup D$ is an edge in $K_n^{(r)}$. Since $c$ is small enough, applying the pigeonhole principle shows that there are at least 
\[\frac{\lfloor\frac{m}{r-d}\rfloor}{k^{2(r-d)}}\geq\frac{\frac{m}{r-d}-1}{(cn^{\frac{1}{2(r-d)+1}})^{2(r-d)}} \geq t-2\]
elements in $\mathcal{Q}$ such that these $Q_i\cup D$ are of the same color under $f_z$ for all vertex $z\in (e\cup e')\backslash D$. Without loss of generality, let $Q_1,Q_2,..,Q_{t-2}$ are $(t-2)$ such $(r-d)$-sets. So, $S'=\{e,e',Q_1\cup D,Q_2\cup D,...,Q_{t-2}\cup D\}$ is a copy of $S_r(d,t)$ such that no coloring $f_v$ corresponding to some vertex $v\in V(S')$ makes $S'$ rainbow.
\end{proof}

\subsection{Special paths in $3$-graph} \label{sec:SP3}

We show that $C_3(n,{\rm SP}_3)$ increases at a rate comparable to $(\frac{\log n}{\log \log n})^{\frac{1}{8}}$. 
This finding is also pertinent to Problem \ref{originalproblem}.

\begin{proof}[Proof of \autoref{STP3}]
We first prove $C_3(n,{\rm SP}_3)=\Omega((\frac{\log n}{\log \log n})^{\frac{1}{8}})$. For any $k>0$, let $n \ge 2k^{(k^4+1)^2}+k^4+3$. It is sufficient to show that for every collection of $n$ edge-colorings $f_v: E(K^{(3)}_n) \to [k]$ on $K_n^{(r)}$,  there exists a copy $T$ of ${\rm SP}_3$ in 
 $K^{(3)}_n$ satisfying that $f_u$ is not a rainbow edge-coloring of $T$ for every $u\in V(T)$.  

Let $Y=\{y_1,y_2,...,y_{k^4+1}\}$ be a subset of $V(K_n^{(3)})$ with $|Y|=k^4+1$. We denote by $X$ the vertex set $V(K^{(3)}_n)\setminus Y$. Let $\mathcal{Z}$ be a collection of all disjoint 2-sets of $X$ and the size of $\mathcal{Z}$ is $\left\lfloor {|X|}/{2} \right\rfloor \ge k^{(k^4+1)^2}+1$. Note that for every $y\in Y$ and $z\in \mathcal{Z}$, $yz:=\{y\}\cup z$ forms an edge in $K^{(3)}_n$. 

For every 2-set $z\in \mathcal{Z}$, consider the vectors 
\begin{align*}
\begin{gathered}
(f_{y_1}(y_1z),f_{y_1}(y_2z),...,f_{y_1}(y_{k^4+1}z),
f_{y_2}(y_1z),f_{y_2}(y_2z),...,f_{y_2}(y_{k^4+1}z),...,\\ 
f_{y_{k^4+1}}(y_1z),f_{y_{k^4+1}}(y_2z),...,f_{y_{k^4+1}}(y_{k^4+1}z)),
\end{gathered}
\end{align*}
and each of them has $(k^4+1)^2$ coordinates. Note that the value of $f_y(zy')$ has $k$ different choices for any two elements $y,y'\in Y$. By the pigeonhole principle, there exists $z_1:=\{x_1,x_2\}, z_2:=\{x_3,x_4\}\in \mathcal{Z}$ such that $z_1,z_2$ have the same vectors. Consider the colorings $f_{x_1},f_{x_2},f_{x_3}$ and $f_{x_4}$. Applying the pigeonhole principle again, there are two vertices $y_1,y_2\in Y$ such that $f_{x_i}(x_1x_2y_1)=f_{x_i}(x_1x_2y_2)$ for $i=1,2,3,4$. We find a $T=\{x_1x_2y_1,x_1x_2y_2,x_3x_4y_2\}$ such that $f_u$ is not a rainbow edge-coloring of $T$ for every $u\in V(T)$.

Next, we give the proof of lower bound of $C_3(n, \rm{SP_4^1})$. Let $k:= cn^{\frac{1}{8}}$, where the constant $c>0$ is sufficiently small. For any set of $3n$ $k$-edge-colorings of $K_{3n}^{(3)}$, it suffices to find a copy $T$ of $\rm{SP_4}$ such that $T$ is not rainbow for every edge coloring assigned to $V(T)$. 

We partition the vertex set of $K_{3n}^{(3)}$ into two parts $A$ and $B$ with $|B|=2|A|=2n$. Let $\mathcal{B}$ be a collection of all disjoint $2$-sets of $B$ and the size of $\mathcal{B}$ is $n$. Consider the number of quadruples $(a,e_1,e_2,e_3)\in A \times \mathcal{B} \times \mathcal{B}\times \mathcal{B}$ such that $f_a(ae_1)=f_a(ae_2)=f_a(ae_3)$. Let $f_a^{-1}(i)$ be the set of elements $e\in \mathcal{B}$ such that $f_a(ae)=i$, but note that there might be some misuse. Observe that for each element $a$ in $A$, there exists a color $i$ such that $a$ contributes at least 
$$\underset{i=1}{\overset{cn^{\frac{1}{8}}}{\sum}} \binom{|f_a^{-1}(i)|}{3}\ge cn^{\frac{1}{8}}\cdot \binom{|\mathcal{B}|/cn^{\frac{1}{8}}}{3}\ge \frac{n^{\frac{11}{4}}}{7c^2}$$
many such quadruples by the convexity of the function $\binom{x}{3}$. Moreover, since there are $n$ choices for $a\in A$, totally there are at least $\frac{n^{\frac{15}{4}}}{7c^2}$ such quadruples of $(a,e_1,e_2,e_3)\in A \times \mathcal{B} \times \mathcal{B}\times \mathcal{B}$. By pigeonhole principle, there are three distinct $2$-sets $e_1, e_2,e_3$ in $\mathcal{B}$ such that there are at least $(\frac{n^{15/4}}{7c^2})/n^3=\frac{n^{3/4}}{7c^2}$ many vertices $a\in A$ satisfying $f_a(ae_1)=f_a(ae_2)=f_a(ae_3)$, and let $A'$ be the subset of $A$ satisfying the above property i.e., for any $a\in A'$, we have $f_a(ae_1)=f_a(ae_2)=f_a(ae_3)$. Then we have $|A'|= \frac{n^{3/4}}{7c^2}$. Let $e_i=u_iv_i$ for each $i\in [3]$. Then we consider the $k^6$-edge-coloring on $A'\cup \{u_1,u_2,u_3,v_1,v_2,v_3\}$ such that 
\[
f(ae_2)=(f_{u_1}(ae_2), f_{u_2}(ae_2), f_{u_3}(ae_2), f_{v_1}(ae_2), f_{v_2}(ae_2), f_{v_3}(ae_2))
\]
for all $a\in A'$ and the other edges we do not care.
By pigeonhole principle, there are at least $(\frac{n^{3/4}}{7c^2})/k^6\ge 2$ elements in $A'$, say $a_1$ and $a_2$, such that $f_v(a_1e_2)=f_v(a_2e_2)$ for any $v\in \{u_1, u_2, u_3, v_1, v_2, v_3\}$. Then we find a copy $T$ of $\rm{SP_4^1}$ on vertex set $\{a_1, a_2, u_1, u_2, u_3, v_1, v_2, v_3\}$ with the following properties:

\begin{itemize}
    \item $f_{a_1}(a_1e_1)=f_{a_1}(a_1e_2)$; $f_{a_2}(a_2e_2)=f_{a_2}(a_2e_3)$;
    \item for any $v\in \{u_1, u_2, u_3, v_1, v_2, v_3\}$, we have $f_v(a_1e_2)=f_v(a_2e_2)$.
\end{itemize}

By applying the same operation---specifically, treating $A$ as a collection of mactchings and $B$ as collection of vertices---we can obtain a copy of $\rm{SP_4^2}$ on vertex set $\{a_1, a_2, a_3, u_1, u_2, v_1, v_2\}$ with the following properties:
\begin{itemize}
    \item $f_{a_1}(a_2u_1u_2)=f_{a_1}(a_2v_1v_2)$; $f_{a_2}(a_2u_1u_2)=f_{a_2}(a_2v_1v_2)$; $f_{a_3}(a_2u_1u_2)=f_{a_3}(a_2v_1v_2)$;
    \item $f_{u_1}(a_1u_1u_2)=f_{u_1}(a_2u_1u_2); f_{u_2}(a_1u_1u_2)=f_{u_2}(a_2u_1u_2); f_{v_1}(a_2v_1v_2)=f_{v_1}(a_3v_1v_2);\\ f_{v_2}(a_2v_1v_2)=f_{v_2}(a_3v_1v_2)$.
\end{itemize}

In order to bound $C_3(n, {\rm SP^1_t})$ and $C_3(n, {\rm SP^2_t})$ for $t\ge 5$, we need the following claim.

\begin{claim}\label{rainbow H}
Let $h,r$ be two positive constants. For an $r$-graph $H$ with $h$ vertices, if there exist $4$ edges $e_1,e_2,e_3,e_4$ in $H$ satisfying $|e_1\cap e_2|=d>0$ and $e_i\cap (e_1\cup e_2)=\emptyset$ for $i=3,4$, then $C_r(n,H)= \Omega(n^{\frac{1}{2(r-d)+1}})$.
\end{claim}
\begin{proof}
Let $k:=cn^{\frac{1}{2(r-d)+1}}$ where the constant $c>0$ is sufficiently small.
We consider the complete $r$-graph $K_n^{(r)}$ and any set of $n$ $k$-edge-colorings of $K_n^{(r)}$. It suffices to find a copy of $H$ such that $H$ is not rainbow for any $k$-edge-coloring assigned to $V(H)$. Given a $d$-vertex set $D:=\{v_1,...,v_d\}\subseteq V(K_n^{(r)})$, we consider the sunflower $S:=S_r(d,\lfloor\frac{n-d}{r-d}\rfloor)$ in $K_n^{(r)}$ with its core $D$. 
By the pigeonhole principle, for each $i\in[d]$,  there exists an edge set $E_i\subseteq  E(S)$ of size at least $\lfloor \frac{\lfloor\frac{n-d}{r-d}\rfloor}{k} \rfloor\ge \frac{1}{c}\cdot n^{1-\frac{1}{2(r-d)+1}}$ such that $f_{v_i}(e')=f_{v_i}(e'')$ for any $e',e''\in E_i$. 
Therefore, we apply the pigeonhole principle to $E(S)$ $d$ times and obtain a edge set $E$ of size at least $\lfloor \frac{\lfloor\frac{n-d}{r-d}\rfloor}{k^d} \rfloor\ge \frac{1}{c^d}\cdot n^{1-\frac{d}{2(r-d)+1}}$ such that $f_{v_i}(e')=f_{v_i}(e'')$ for any $e',e''\in E$ and $i\in [d]$.
Let $P=V(G)\setminus V(S)$. 
Let $f^{-1}_p(i)$ be the set of edges assigned $i$ for the coloring $f_p$. The number of triples $(p,e',e'')\in P \times E \times E$ with $f_p(e')=f_p(e'')$ is at least
$$\underset{p\in P}\sum \underset{i=1}{\overset{k}\sum}\binom{|f_p^{-1}(i)|}{2}\ge \underset{p\in P}\sum k\binom{|E|/k}{2}\ge c'\frac{n^3}{k^{2d+1}}$$
with a sufficiently small positive constant $c'$.

By pigeonhole principle, there exists a pair of edges $e_1,e_2\in E$ and a subset $A\subseteq P$ with $|A|\ge \frac{n^3}{k^{2d+1}}/\binom{|E|}{2} \ge \frac{n}{ck}$ such that $f_a(e_1)=f_a(e_2)$ for each $a\in A$. Let $t:=|e^3\cap e^4|$ and $T:=\{a_1,...,a_t\}\subseteq A$. Let $S':=S_r'(t,\lfloor \frac{|A|-t}{r-t} \rfloor)$ be a sunflower in $A$ with its core $T$. We consider the coloring $f_b$ on $S'$ for any $b\in (e_1\cup e_2)\setminus D$. Since $k=cn^{\frac{1}{2(r-d)+1}}$, applying pigeonhole principle shows that there exists an edge subset $E'\subseteq E(S_r'(t,\lfloor \frac{|A|-t}{r-t} \rfloor))$ of size at least $\frac{\lfloor \frac{|A|-t}{r-t} \rfloor}{k^{2(r-d)}}\ge 2$ such that $f_b(e')=f_b(e'')$ for any $b\in e_1\cup e_2 \setminus D$ and any $e',e''\in E'$. 
Therefore, we find an $H$ copy which is not rainbow for any coloring $f_u$ with $u
\in V(H)$.
\end{proof}

Given that the hypergraphs $\text{SP}^1_t$ and $\text{SP}^2_t$ with $t \geq 5$ incorporate the four edges specified in Claim~\ref{rainbow H}, it follows that $C_3(n, \text{SP}^1_t) = \Omega(n^{\frac{1}{3}})$ and $C_3(n, \text{SP}^2_t) = \Omega(n^{\frac{1}{3}})$ for all $t \geq 5$.
This completes the proof.
\end{proof}

\begin{remark}
Claim \ref{rainbow H} is also of independent interest. Let $\rm{LP}_t$ denote the $3$-uniform loose path with $t$ edges and ${\rm TP}_t$ represent the $3$-uniform tight path with $t$ edges. By extending the same reasoning framework and invoking Claim~\ref{rainbow H}, we establish that $C_3(n, {\rm LP}_t) = \Omega(n^{\frac{1}{5}})$ for $t \geq 5$ and $C_3(n, {\rm TP}_t) = \Omega(n^{\frac{1}{3}})$ for $t \geq 6$.

Furthermore, we introduce a unified notion generalizing both ${\rm LP}_t$ and ${\rm TP}_t$. 
Given a set $V:=\{v_1,v_2,\ldots,v_n\}$, we set $e_i:=\{v_i,v_{i+1},\ldots,v_{i+r-1}\}~(1\leq i\leq t-r+1)$ and let $\mathcal{I}$ be a family such that each $I\in \mathcal{I}$ is a sequentially ordered subset of the ordered set $\{1,2,\ldots,t-r+1\}$, satisfying that $\{1,t-r+1\}\subseteq I$ and the difference between any two consecutive elements in $I$ does not exceed $r-1$. 

Let $\mathcal{P}^{(r)}_t$ be a family of $r$-uniform hypergraphs such that each hypergraph in $\mathcal{P}^{(r)}$ has vertex set $V$ and edge set $E=\{e_i~|~i\in I\}$ for some $I\in \mathcal{I}$ with $|I|=t$. Given ${\rm P}^{(r)} \in \mathcal{P}^{(r)}_t$, if there exists a pair of distinct edges $e,e'\in E$ such that $e({\rm P}^{(r)}[V\setminus (e\cup e')]\ge 2$, then select such a pair that maximizes the intersection $|e\cap e'|$, and denote this maximum value by $\gamma({\rm P}^{(r)})$; and if such a pair does not exists, then we set $\gamma({\rm P}^{(r)})=-\infty$. For example, we have ${\rm LP}_t,{\rm TP}_t\in \mathcal{P}^{(3)}_t$ and moreover,
$\gamma({\rm LP}_t)=1$ if $t\geq 5$ and $\gamma({\rm TP}_t)=2$ if $t\geq 6$.

Now, by Claim \ref{rainbow H}, we have that 
\[C_r(n,{\rm P}^{(r)})=\Omega(n^{\frac{1}{2(r-\gamma({\rm P}^{(r)}))+1}})\]
for each ${\rm P}^{(r)} \in \mathcal{P}^{(r)}_t$ and each $t$.
\end{remark}

\section{Concluding remarks}
Indeed, in Theorem \ref{sunflowers}, we obtain that for every 3-graph $H$ with at least 163 edges, there is a constant $b=b(H)>0$ such that $C_3(n,H)=\Omega(n^b)$. 
It is meaningful to decrease this bound of the number of edges.

In Section \ref{upperbounds}, we show the lower bounds for the local rainbow coloring numbers of sunflowers, matchings, two types of special paths---loose paths with length at least 5 and tight paths with length at least 6. While for loose path with length 4 and tight paths with length $t\in\{3,4,5\}$, we are unable to provide an elegant bound.

As the 3-graph ${\rm TP_3}$ is a subgraph of ${\rm TC_e}$, we can also get an upper bound for ${\rm TP_3}$ 
%but it is slightly worse than \autoref{TP3}
, which is $C_3(n,{\rm TP_3})\leq C_3(n,{\rm TC_e})^6=n^{o(1)}$. For the upper bounds of another two 3-graphs ${\rm SP_3}$ and ${\rm LC_3}$ related to Problem \ref{originalproblem}, we do not know if they are subpolynomial. This can be proved either by directly giving their upper bound constructions separately or by giving the upper bound construction of a new 3-graph containing both ${\rm SP_3}$ and ${\rm LC_3}$, such as the following 3-graph ${\rm LC_e}$. 
\begin{figure}[H]
    \centering
    \includegraphics[width=0.25\linewidth]{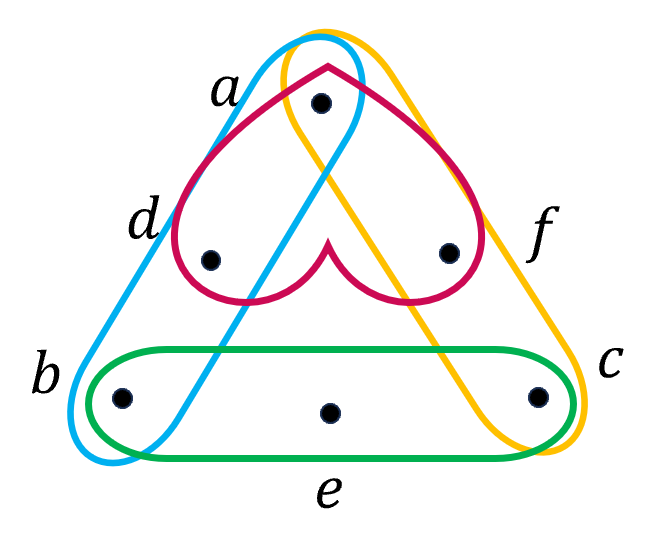}
    \caption{A loose cycle with an added edge $adf$, abbreviated as ${\rm LC_e}$}
    \label{fig:enter-label}
\end{figure}

\section*{Acknowledgments}

The authors would like to express their gratitude to Wei Liu for her assistance with the Python code used in the proof of Theorem~\ref{3constantupperbound}.

\bibliographystyle{plain}
\bibliography{ref.bib}
\appendix
\section{2-locally large}
Indeed, we just need to prove all 3-graphs with exact 4 edges are 2-locally large. By Property \ref{property 2-locally large}, all 3-graphs with at least 4 edges are 2-locally large. The Python code in \ref{fulu1} is to check if all 3-graphs with 4 edges are 2-locally large.
In addition, the Python code in \ref{fulu2} determines whether all graphs with 3 edges are 2-locally and outputs the 2-locally cases.

\subsection{All 3-graphs with 4 edges are 2-locally large} \label{fulu1}
\begin{lstlisting}[style=Python, title=``Python Code for 3-graphs with 4 edges"]                
import itertools
from functools import lru_cache

# Check if x is the i-th largest element in S
@lru_cache(maxsize=None)
def ith(x, S, i):
    # Sort S in descending order
    sorted_S = tuple(sorted(S, reverse=True))
    try:
        index = sorted_S.index(x)
        return index + 1 == i
    except ValueError:
        return False

# Compute graph features for fast non - isomorphism check
@lru_cache(maxsize=None)
def graph_features(H):
    # Sorted edge lengths
    edge_lengths = tuple(sorted(len(edge) for edge in H))
    # Dictionary to store vertex degrees
    vertex_degrees = {}
    for edge in H:
        for vertex in edge:
            vertex_degrees[vertex] = vertex_degrees.get(vertex, 0) + 1
    # Sorted vertex degree sequence
    degree_sequence = tuple(sorted(vertex_degrees.values()))
    return edge_lengths, degree_sequence

# Check if H1 and H2 are isomorphic
@lru_cache(maxsize=None)
def is_isom(k, H1, H2):
    # Compare graph features first
    feats1 = graph_features(tuple(map(tuple, H1)))
    feats2 = graph_features(tuple(map(tuple, H2)))
    if feats1 != feats2:
        return False
    for gamma in itertools.permutations(range(k), k):
        gamma_map = {i + 1: gamma[i] + 1 for i in range(k)}
        transformed_H2 = [tuple(sorted(gamma_map[v] for v in edge)) for edge in H2]
        sorted_H1 = [tuple(sorted(edge)) for edge in H1]
        if sorted(sorted_H1) == sorted(transformed_H2):
            return True
    return False

# Calculate the Tai set
def calculate_Tai(a, i, r, sigma, e1, se1, e2, se2, e3, se3, e4, se4):
    Tai = []
    if i <= r:
        edge_sets = [(e1, se1), (e2, se2), (e3, se3), (e4, se4)]
    else:
        se1_ = se1 + (sigma[a - 1],)
        se2_ = se2 + (sigma[a - 1],)
        se3_ = se3 + (sigma[a - 1],)
        se4_ = se4 + (sigma[a - 1],)
        edge_sets = [(e1, se1_), (e2, se2_), (e3, se3_), (e4, se4_)]
    for e, se in edge_sets:
        if (a in e) == (i <= r) and ith(sigma[a - 1], se, i if i <= r else i - r):
            Tai.append(e)
    return Tai

# Check if the graph formed by e1, ..., e4 has property P
def hasP(k, r, VH, e1, e2, e3, e4):
    rr = 2 * r + 1
    I = range(1, rr + 1)
    # Check if any vertex is not in any edge
    vertex_counts = {v: 0 for v in VH}
    for edge in [e1, e2, e3, e4]:
        for v in edge:
            vertex_counts[v] += 1
    if 0 in vertex_counts.values():
        return False
    for sigma in itertools.permutations(VH, k):
        se1 = tuple(sigma[v - 1] for v in e1)
        se2 = tuple(sigma[v - 1] for v in e2)
        se3 = tuple(sigma[v - 1] for v in e3)
        se4 = tuple(sigma[v - 1] for v in e4)
        gooda_count = 0
        for a in VH:
            for i in I:
                Tai = calculate_Tai(a, i, r, sigma, e1, se1, e2, se2, e3, se3, e4, se4)
                if len(Tai) > 1:
                    gooda_count += 1
                    break
            # Early pruning
            if gooda_count + k - a < k:
                break
        if gooda_count == k:
            return sigma
    return False

# Main program
r = 3
K = [4, 5, 6, 7, 8, 9, 10, 11, 12]
badH = []
goodH = []
E = 4
e1 = list(range(1, r + 1))

for k in K:
    VH = list(range(1, k + 1))
    redges_set = list(itertools.combinations(VH, r))
    # Filter out invalid edges
    valid_edges = [edge for edge in redges_set if r + 1 <= max(edge) <= 3 * r]
    for e2, e3, e4 in itertools.combinations(valid_edges, E - 1):
        e2 = list(e2)
        e3 = list(e3)
        e4 = list(e4)
        Ue = list(set(e1).union(set(e2)).union(set(e3)).union(set(e4)))
        if len(Ue) == k and r + 1 <= max(e2) <= 2 * r and r + 1 <= max(e3) <= 3 * r:
            result = hasP(k, r, VH, e1, e2, e3, e4)
            H2 = [e1, e2, e3, e4]
            if not result:
                exists = any(is_isom(k, tuple(map(tuple, existing_graph)), tuple(map(tuple, H2))) for existing_graph in badH)
                if not exists:
                    badH.append(H2)
                    print(e1, e2, e3, e4, ", bad ! ! !")
                else:
                    print(e1, e2, e3, e4, ", badH exists. ")
            else:
                exists = any(is_isom(k, tuple(map(tuple, existing_graph)), tuple(map(tuple, H2))) for existing_graph in goodH)
                if not exists:
                    goodH.append(H2)
                    print(e1, e2, e3, e4, ", σ =", result)
                else:
                    print(e1, e2, e3, e4, ", goodH exists. ")

print("all bad H:", badH)
print("the NO. of graphs without P:", len(badH))
print("all good H:", goodH)
print("the NO. of graphs with P:", len(goodH))
\end{lstlisting}

\subsection{Not all 3-graphs with 3 edges are 2-locally large}  \label{fulu2}

\begin{lstlisting}[style=Python, title=``Python Code for 3-graphs with 3 edges"]        
import itertools

def ith(x, S, i):
    #"Verify if x is the i-th largest element in S."
    sorted_S = sorted(S, reverse=True)
    return sorted_S.index(x) + 1 == i

def is_isomorphic(H1, H2, k):
    #"Check if hypergraphs H1 and H2 are isomorphic by trying all permutations of vertices."
    VH = list(range(1, k + 1))
    for gamma in itertools.permutations(VH, k):
        # Apply permutation to H2 and compare with H1
        perm_H2 = [[gamma[vertex - 1] for vertex in edge] for edge in H2]
        if all(sorted(edge) == sorted(perm_edge) for edge, perm_edge in zip(H1, perm_H2)):
            return True
    return False

def has_property(k, r, VH, e1, e2, e3, e4):
    #"Check if the hypergraph defined by edges e1, e2, e3, e4 has property P."
    rr = 2 * r + 1
    I = range(1, rr + 1)
    for sigma in itertools.permutations(VH, k):
        se = lambda edge: [sigma[vertex - 1] for vertex in edge]
        se1, se2, se3, se4 = se(e1), se(e2), se(e3), se(e4)
        
        goodv_count = sum(
            any(
                (a in edge and ith(sigma[a - 1], se(edge), i)) or
                (a not in edge and ith(sigma[a - 1], se(edge) + [sigma[a - 1]], i - r))
                for i in I
            )
            for a in VH for edge in [e1, e2, e3, e4]
        )

        if goodv_count >= k:
            return sigma
    return None

# Main program
k = 4
r = 3
VH = list(range(1, k + 1))
badH = []
goodH = []

redges_set = list(itertools.combinations(VH, r))
for e2, e3, e4 in itertools.combinations(redges_set, 3):
    e1 = tuple(range(1, r + 1))
    Ue = set(e1).union(e2, e3, e4)
    
    if len(Ue) == k and r + 1 <= max(e2) <= 2 * r and r + 1 <= max(e3) <= 3 * r:
        result = has_property(k, r, VH, e1, e2, e3, e4)
        if not result:
            if not any(is_isomorphic(H, [e1, e2, e3, e4], k) for H in badH):
                badH.append([e1, e2, e3, e4])
                print(e1, e2, e3, e4, ", bad ! ! !")
        else:
            print(e1, e2, e3, e4, ", σ =", result)

print("All bad H:", badH)
print("Number of graphs without P:", len(badH))
\end{lstlisting}
\end{document}